\documentclass[10pt]{amsart}

\usepackage{amsmath}
\usepackage{amssymb}
\usepackage{indentfirst}
\usepackage[latin1]{inputenc}
\usepackage{geometry}
\usepackage{amsfonts}
\usepackage{algorithm}
\usepackage{algpseudocode}
\usepackage{cancel}
\algrenewcommand{\algorithmicrequire}{\textbf{Input:}}
\algrenewcommand{\algorithmicensure}{\textbf{Output:}}
\algrenewcommand{\algorithmiccomment}[1]{//#1}
\usepackage[bookmarks=false,colorlinks=true,linkcolor=black,citecolor=black,filecolor=black,urlcolor=black]{hyperref}
\DeclareSymbolFont{largesymbols}{OMX}{yhex}{m}{n}
\DeclareMathAccent{\widehat}{\mathord}{largesymbols}{"62}

\newcommand{\Q}{{\mathbb Q}}
\newcommand{\Z}{{\mathbb Z}}

\newcommand{\N}{{\mathbb N}}

\newcommand{\U}{\mathcal{U}}
\renewcommand{\O}{\mathcal{O}}

\newcommand{\lcm}{\textnormal{lcm}}

\newcommand{\Bass}[3]{u_{#1,#2}(#3)}

\newcommand{\Hc}{{\mathcal{H}}}
\newcommand{\inv}{{^{-1}}}
\newcommand{\GEN}[1]{\left\langle #1 \right\rangle}

\newtheorem{theorem}{Theorem}[section]
\newtheorem{proposition}[theorem]{Proposition}
\newtheorem{corollary}[theorem]{Corollary}
\newtheorem{lemma}[theorem]{Lemma}

\theoremstyle{definition}

\theoremstyle{remark}

\numberwithin{equation}{section}

\begin{document}

\title[Writing units as a product of Bass units]{Writing units of integral group rings of finite abelian groups as a product of Bass units}

\author{Eric Jespers}
\address{Department of Mathematics, Vrije Universiteit Brussel,
Pleinlaan 2, 1050 Brussels, Belgium}
\email{efjesper@vub.ac.be}

\author{\'{A}ngel del R\'io}
\address{Departamento de Matem\'{a}ticas, Universidad de Murcia,  30100 Murcia, Spain}
\email{adelrio@um.es}

\author{Inneke Van Gelder}
\address{Department of Mathematics, Vrije Universiteit Brussel,
Pleinlaan 2, 1050 Brussels, Belgium}
\email{ivgelder@vub.ac.be}

\thanks{The first and second authors have been partially supported by the Ministerio de Ciencia y Tecnolog\'{\i}a of Spain and Fundaci\'{o}n S\'{e}neca of Murcia. The first author is partially supported by Fonds voor Wetenschappelijk Onderzoek Vlaanderen-Belgium and Onderzoeksraad Vrije Universiteit Brussel. The third  author is supported by Fonds voor Wetenschappelijk Onderzoek Vlaanderen-Belgium.
}

\date{\today}



\begin{abstract}
We give a constructive proof of the theorem of Bass and Milnor saying that if $G$ is a finite abelian group then the Bass units of the integral group ring $\Z G$ generate a subgroup of finite index in its units group $\U(\Z G)$. Our proof provides algorithms to represent some units that contribute to only one simple component of $\Q G$ and generate a subgroup of finite index in $\U(\Z G)$ as product of Bass units.
We also obtain a basis $B$ formed by Bass units of a free abelian subgroup of finite index in $\U(\Z G)$ and give, for an arbitrary Bass unit $b$, an algorithm to express $b^{\varphi(|G|)}$ as a product of a trivial unit and powers of at most two units in this basis $B$.
\end{abstract}

\maketitle

\section{Introduction}

Let $G$ be a finite group, $g$ an element of $G$ of order $n$ and $k$ and $m$ positive integers so that $k^m \equiv 1 \mod n$. Then
    $$\Bass{k}{m}{g}=(1+g+\dots + g^{k-1})^{m}+\frac{1-k^m}{n}(1+g+g^2+\dots+g^{n-1})$$
is a unit of the integral group ring $\Z G$. The units of this form were introduced by Bass in \cite{bass1966} and are known as Bass units or Bass cyclic units. Bass proved that if $G$ is a cyclic group then the Bass units of $\Z G$ generate a subgroup of finite index in $\U(\Z G)$. (The group of units of a ring $R$ is denoted $\U(R)$.) Bass and Milnor extended this result to finite abelian groups. In this paper we will refer to this result as the Bass-Milnor Theorem while the Bass Theorem refers to the result for cyclic groups. The Bass Theorem also provides a basis consisting of Bass units for a free abelian subgroup of finite index in $\U(\Z G)$, provided $G$ is cyclic. As far as we know, no basis consisting of Bass units for a free abelian subgroup of finite index is known for an arbitrary abelian group $G$.

It is worth mentioning that in more recent work Marciniak and Sehgal \cite{2005MarciniakSehgal,2006MarciniakSehgal} investigate the so called group of generic units $\textnormal{Gen}(G)$ in $\U(\Z G)$, for $G$ a finite abelian group. It turns out that $\textnormal{Gen}(G)$ is generated by the Hoechsmann units \cite{Hoechsmann1992} (for the definition see \cite[p. 34]{Sehgal1993}), and some power of a Hoechsmann unit is a product of two Bass units. Furthermore, the group generated by the Bass units $\Bass{k}{\varphi(|G|)}{g}$ is contained in $\textnormal{Gen}(G)$. For other results, we refer the reader to \cite{2005MarciniakSehgal,2006MarciniakSehgal}.

In the remainder of the paper we assume that $G$ is a finite abelian group. If $u\in \U(\Z G)$ then, by the Bass-Milnor Theorem, some power of $u$ is a product of Bass units. Unfortunately the proofs by Bass and Bass-Milnor do not provide a method to express some power of $u$ as a product of Bass units. The aim of this paper is to obtain such a method for some relevant units. To do so we first obtain a new proof of the Bass-Milnor Theorem which follows a different approach than the proofs of Bass and Bass-Milnor. In order to have a feeling of the features of this new proof first we recall the main steps of the Bass and Bass-Milnor proofs and outline those of our proof.

Let $\Gamma$ be a finitely generated abelian group. The rank of all free abelian subgroups of finite index in $\Gamma$ is an invariant of the group, called the rank of $\Gamma$. Assume that $\Gamma$ has rank $r$ and let $u_1,\dots,u_r\in \Gamma$. Then $\GEN{u_1,\dots,u_r}$ has finite index in $\Gamma$ if and only if $u_1,\dots,u_r$ are multiplicatively independent.
For example, $\U(\Z G)$ is finitely generated abelian and has rank $r=\frac{1+k_2+|G|-2c}{2}$, where $c$ is the number of cyclic subgroups of $G$ and $k_2$ is the number of elements of $G$ of order $2$. This result is due to \cite{Higman1940} (for the above formula for $r$ see also \cite{1969AyoubAyoub}). Bass took a concrete list of $r$ Bass units of $\Z G$ for $G$ cyclic and proved that they are multiplicatively independent. To do so, Bass used the Bass Independence Theorem which in turn uses the Franz Independence Lemma (see \cite[11.3, 11.8]{Sehgal1993} for details).
Bass and Milnor proved, using K-Theory, that the group generated by the units of integral group rings of cyclic subgroups of $G$ has finite index in $\U(\Z G)$. However, both the independency and the K-Theory arguments are useless when trying to write a given unit as a product of Bass units.

Alternatively, assume that $\Gamma$ is a subgroup of finite index in a finitely generated abelian group $\Lambda$ and that we know a subset $X$ of $\Lambda$ which generates a subgroup of finite index in $\Lambda$. Let $Y$ be a subset of $\Gamma$. Then $\GEN{Y}$ has finite index in $\Gamma$ if and only if for every $x\in X$ there is a positive integer $m$ such that $x^m \in \GEN{Y}$. In our proof we take $\Gamma=\U(\Z G)$, $Y$ the set of Bass units of $\Z G$, $\Lambda=\U(\O)$, where $\O$ is the unique maximal order of $\Q G$, and $X$ the set of cyclotomic units of $\Q G$, which are defined as follows.

Let $\zeta_n$ denote a complex root of unity of order $n$. If $n>1$ and $k$ is an integer coprime to $n$ then
    $$\eta_{k}(\zeta_n)=\frac{1-\zeta_n^k}{1-\zeta_n}=1+\zeta_n+\zeta_n^2+\dots+\zeta_n^{k-1}$$
is a unit of $\Z[\zeta_n]$, the $n$-th cyclotomic ring of integers. We extend this notation by setting
    $$\eta_k(1)=1.$$
The units of the form $\eta_k(\zeta_n^j)$, with $j,k$ and $n$ integers such that $\gcd(k,n)=1$, are called the cyclotomic units of $\Q(\zeta_n)$.
A classical result, which goes back to Kummer, states that the cyclotomic units of $\Q(\zeta_n)$ generate a subgroup of finite index in $\U(\Z[\zeta_n])$ \cite{1982Washington}.

By a well-known theorem of Perlis and Walker, $\Q G$ is isomorphic to a direct product of cyclotomic fields. For simplicity, we consider this isomorphism as an equality $\Q G=\prod_{i=1}^k \Q(\zeta_{n_i})$. Then $\Z G$ is an order of $\Q G$ and $\O=\prod_{i=1}^k \Z[\zeta_{n_i}]$ is the unique maximal order of $\Q G$. In particular, $\Z G \subseteq \O$. Moreover, $\Gamma=\U(\Z G)$ has finite index in $\Lambda=\U(\O)$ (see \cite[Lemma~4.6]{Sehgal1993}). The cyclotomic units of $\Q G$ are, by definition, the elements of $\Q G$ which project to a cyclotomic unit of $\Q(\zeta_{n_i})$ for some $i=1,\dots,k$ and project to $1$ in the remaining components. Having in mind that the cyclotomic units of $\Q(\zeta_{n_i})$ generate a subgroup of finite index in $\U(\Z[\zeta_{n_i}])$, we conclude that the set $X$ of cyclotomic units of $\Q G$ generates a subgroup of finite index in $\Lambda$. Thus the Bass-Milnor Theorem is equivalent to the following proposition.

\begin{proposition}\label{OneComponentBass}
Let $G$ be a finite abelian group and let $\Q G = \prod_{i=1}^k \Q(\zeta_{n_i})$, the realization of Perlis-Walker Theorem.
Then for every cyclotomic unit $u$ of $\Q G$ there is a positive integer $m$ such that $u^m$ is a product of Bass units of $\Z G$.
\end{proposition}

The proof of Proposition~\ref{OneComponentBass} (actually of Lemma~\ref{BassProjectionsL}) is constructive and avoids K-Theory and independence arguments. As a result the proof provides an algorithm that for a cyclotomic unit $\eta$ as input, returns $m$ and an expression of $\eta^m$ as a product of Bass units (see Algorithm~\ref{alg1}). This is the first result of the paper. The second result consists in giving a concrete basis $B$ formed by Bass units for a free abelian subgroup of finite index in $\U(\Z G)$. The proof of the fact that $B$ is a basis also provides an algorithmic method to write some power of any arbitrary Bass unit as a product of elements of $B$ (see Section~\ref{SectionBasis}).

Observe that our result provides an algorithmic method to write some power of a given $u\in \U(\O)$ as a product of Bass units, or even of elements of $B$, as far as we can write some power of each projection of $u$ in the simple components of $\Q G$ as a product of cyclotomic units.
However we do not know of any algorithmic method that starting from an element $u\in \U(\Z[\zeta_n])$ as input, returns an expression of some power of $u$ as a product of cyclotomic units.

\section{A new proof of the Bass-Milnor Theorem}

Throughout the rest of the paper $G$ is a finite abelian group. We denote the group generated by the Bass units of $\Z G$ by $B(G)$.
In this section we prove Proposition~\ref{OneComponentBass}. This provides a new proof of the Bass-Milnor Theorem which states that $B(G)$ has finite index in $\U(\Z G)$.

First of all we obtain a precise realization of Perlis and Walker Theorem which states that there is an isomorphism $f:\Q G \rightarrow \prod_d \Q(\zeta_d)^{k_d}$, where $k_d$ denotes the number of cyclic subgroups of $G$ of order $d$. This isomorphism is realized as follows. Let $\Hc=\Hc(G)$ denote the set of subgroups $H$ of $G$ such that $G/H$ is cyclic. For every subgroup $H\in \Hc$ we fix a linear representation $\rho_H$ of $G$ with kernel $H$. We also denote by $\rho_H$ the linear extension  of $\rho_H$ to $\Q G$. If $d=[G:H]$ then $\rho_H(\Q G)=\Q(\zeta_d)$, where $\zeta_{d}$ denotes a primitive $d$-th root of unity. Then
    $$f=\prod_{H\in \Hc} \rho_H:\Q G \longrightarrow \prod_{H\in \Hc} \Q(\zeta_{[G:H]})$$
is an isomorphism. This isomorphism is the same as the one of Perlis and Walker \cite{1950PerlisWalker} (see for example also \cite{2002SehgalMilies}), since $k_d$ equals the number of subgroups $H\in \Hc$ such that $[G:H]=d$.

We will use the following equalities (the order of an element $g\in G$ is denoted by $|g|$)
    \begin{eqnarray}\label{BassCyclotomic}
    \rho_H(\Bass{k}{m}{g}) &=& \eta_k(\rho_H(g))^m, \quad (H\in \Hc, g\in G, k^m\equiv 1 \mod |g|), \\
    \prod_{i=0}^{n-1} (1-X\zeta_n^i) &=& 1-X^n. \label{1-Xn}
    \end{eqnarray}

Let $\xi$ be a root of unity and assume that $k$ is coprime to $n$ and the order of $\xi$. If $\xi^n\ne 1$ then, using (\ref{1-Xn}), we obtain
    \begin{eqnarray*}
    \prod_{i=0}^{n-1} \eta_k(\xi \zeta_n^i) = \prod_{i=0}^{n-1} \frac{1-\xi^k \zeta_n^{ki}}{1-\xi{\zeta_n}^i} =
    \frac{\prod_{i=0}^{n-1} (1-\xi^k \zeta_n^i)}{\prod_{i=0}^{n-1} (1-\xi{\zeta_n}^i)} = \frac{1-\xi^{kn}}{1-\xi^n} = \eta_k(\xi^n).
    \end{eqnarray*}
Otherwise, i.e. if $\xi^n=1$, then $\xi\zeta_n^j=1$ for some $j=0,1,\dots,n-1$. Then, using that $k$ is coprime to $n$, we deduce that
    \begin{eqnarray*}
    \prod_{i=0}^{n-1} \eta_k(\xi \zeta_n^i) = \prod_{i=0,i\neq j}^{n-1} \frac{1-\zeta_n^{k(i-j)}}{1-\zeta_n^{i-j}} =
    \frac{\prod_{i=1}^{n-1} (1-\zeta_n^{ki})}{\prod_{i=1}^{n-1} (1-\zeta_n^i)} = 1 = \eta_k(\xi^n).
    \end{eqnarray*}
This proves the following equality for every primitive $l$-th root of unity $\xi$
    \begin{equation}\label{ProductCyclotomic}
    \prod_{i=0}^{n-1} \eta_k(\xi \zeta_n^i) =
    \eta_k(\xi^n) \quad (\gcd(k,nl)=1).
    \end{equation}

\begin{lemma}\label{ProductBassCyclic}
Let $g\in G$, $H\in \Hc$ and $K$ be an arbitrary subgroup of $G$. Let $h=|H\cap K|$, $t=[K:H\cap K]$ and let $k$ and $m$ be positive integers such that $(k,t)=1$ and $k^m \equiv 1 \mod |gu|$ for every $u\in K$.
Then
    $$\prod_{u\in K} \rho_H(\Bass{k}{m}{gu}) = \eta_k(\rho_H(g)^t)^{mh}.$$
\end{lemma}

\begin{proof}
As $H=\ker(\rho_H)$, if $u$ runs through the elements of $K$ then $\rho_H(u)$ runs through the $t$-th roots of unity and each $t$-th root of unity is obtained as $\rho_H(u)$ for precisely $h$ elements $u$ of $K$. Therefore
    $$\prod_{u\in K} \rho_H(\Bass{k}{m}{gu}) = \left( \prod_{u\in K}
    \eta_k(\rho_H(g)\rho_H(u)) \right)^{m} = \left(\prod_{i=0}^{t-1} \eta_k(\rho_H(g)\zeta_t^i) \right)^{mh}
     = \eta_k(\rho_H(g)^t)^{mh}$$
as desired. We have used (\ref{ProductCyclotomic}) in the last equality.
\end{proof}

Proposition~\ref{OneComponentBass} is a consequence of the following stronger lemma.

\begin{lemma}\label{BassProjectionsL}
Let $H\in \Hc$ with $d=[G:H]$ and let $k,j\in \N$ be such that $k$ is coprime to $d$. Set $\eta=\eta_k(\zeta_d^j)$ and let $B_k(G)$ be the subgroup of $\U(\Z G)$ generated by the Bass cyclic units of the form $\Bass{k}{m}{g}$ with $g\in G$ and $k^m \equiv 1 \mod |g|$. Then there is a positive integer $m$ and $b\in B_k(G)$ such that $\rho_H(b)=\eta^m$ and $\rho_K(b)=1$ for every $K\in \Hc\setminus \{H\}$.
\end{lemma}
\begin{proof}
Without loss of generality, we may assume that $k$ is coprime to $n=|G|$. Indeed, by an easy Chinese Remainder argument there is an integer $k'$ coprime to $n$ such that $k\equiv k'\mod d$. Then clearly $\eta_k(\zeta_d^j)=\eta_{k'}(\zeta_d^j)$.

We argue by a double induction, first on $n$ and second on $d$. The cases $n=1$ and $d=1$ are trivial. We denote by $P(G,H)$ the statement of the lemma for a finite abelian group $G$ and an $H\in \Hc(G)$. Hence the induction hypothesis includes the following statements:

\begin{itemize}
\item[(IH1):] $P(M,Y)$ holds for every proper subgroup $M$ of $G$ and any $Y\in \Hc(M)$.
\item[(IH2):] $P(G,H_{1})$ holds for every $H_{1}\in \Hc(G)$ with $[G:H_{1}]<[G:H]=d$.
\end{itemize}

We consider two cases, depending on whether $j$ is coprime to $d$ or not.

\textsc{Case 1}: $j$ is not coprime to $d$. Let $p$ be a common prime divisor of $d$ and $j$. Then $H$ is contained in a subgroup $S$ of $G$ with $[G:S]=p$ and $\zeta_{d}^{j}=\zeta_{d'}^{j/p}$ with $d'=[S:H]$. For every $K\in \Hc(G)$, let $\lambda_K$ denote the restriction of $\rho_K$ to $\Q S$. Clearly $\lambda_K$ is the $\Q$-linear extension of a linear representation of $S$ with kernel $S\cap K$. Since $S/(S\cap K)\cong KS/K$ and $KS/K$ is a subgroup of $G/K$ we deduce that $S/(S\cap K)$ is cyclic. Thus $K\rightarrow S\cap K$ defines a map $\Hc(G)\rightarrow \Hc(S)$. This map is  surjective, but maybe not injective. Indeed, let $K_1\in \Hc(S)$. If $K_1\in \Hc(G)$ then clearly the map associates $K_1$ with $K_1$. Otherwise $p$ divides $[G:K_1]$ and $S/K_1$ is a cyclic subgroup of $G/K_1$ of maximal order. This implies that $G/K_1=S/K_1\times L/K_1$ for some subgroup $L$ of $G$ containing $K_1$ and so that $[L:K_1]=p$. Then $G/L\cong S/K_1$, so that $L\in \Hc(G)$, and $L\cap S=K_1$.

Therefore $\Hc(S)=\{K\cap S  :  K\in \Hc(G)\}$. For every $Y\in \Hc(S)$ we choose a $K_Y\in \Hc(G)$ such that $K_Y\cap S=Y$ in such a way that $K_Y=Y$ if $Y\in \Hc(G)$. Then
$$\prod_{Y\in \Hc(S)} \lambda_{K_Y} : \Q S \rightarrow \prod_{Y\in \Hc(S)} \Q (\zeta_{[S:Y]})$$
is an isomorphism of algebras. By the first induction  hypothesis (IH1) there is $b\in B_k(S)$ such that $\lambda_H(b)=\eta^m$ for some positive integer $m$ and $\rho_K(b)=\lambda_K(b)=1$ if $K\in \Hc(G)$ with $K\cap S\ne H$. If $K\in \Hc(G)$ satisfies $K\cap S=H$ then either $K=H$ or $K=H_1$, where $H_1/H$ is the only subgroup of $G/H$ of order $p$, since $G/H$ is cyclic. Moreover $\rho_{H_1}(b)$ is a product of cyclotomic units, by (\ref{BassCyclotomic}). By the second induction hypothesis (IH2) there is $c\in B_k(G)$ such that $\rho_K(c)=1$ for every $K \in \Hc(G)\setminus \{H_1\}$ and $\rho_{H_1}(c)=\rho_{H_1}(b)^{m_1}$ for some positive integer $m_1$. Therefore $\rho_H(b^{m_1}c\inv)=\eta^{mm_1}$ and $\rho_K(b^{m_1}c\inv)=1$ for every $K\in \Hc(G)\setminus \{H\}$. This finishes the proof for this case.

\textsc{Case 2}: $j$ is coprime to $d$. Then $G=\GEN{a,H}$ and $\rho_H(a)=\zeta_d^j$ for some $a\in G$.  As $k$ is coprime to $n$, there is a positive integer $m$ such that $k^m\equiv 1 \mod |au|$ for every $u\in H$.
Hence
    $$\eta^m = \rho_H(\Bass{k}{m}{a}),$$
by (\ref{BassCyclotomic}). Let $$b=\prod_{h\in H} \Bass{k}{m}{ah}.$$ For every $K\in \Hc(G)$, set $$d_K=[G:K],\; d'_K=[G:\GEN{a,K}],\; h_K=|H\cap K| \mbox{ and } t_K=[H:H\cap K].$$ Then, by Lemma \ref{ProductBassCyclic},
    $$\rho_K(b) = \eta_k(\rho_K(a)^{t_K})^{m\, h_K} = \eta_k(\zeta_{d_K}^{t_Kd'_Ku_K})^{m\, h_K},$$
for some integer $u_K$ coprime to $d_K$. If $t_Kd'_K$ is not coprime to $d_K$ then, by Case 1, there is $b_K\in B_k(G)$ such that $\rho_K(b_K)=\rho_K(b)^{m_K}$ for some integer $m_K$ and $\rho_{K_{1}}(b_{K})=1$ for $K_{1}\in \Hc(G)\setminus \{ K \}$. By (IH2), the same holds if $d_K<d$. Let
    $$\Hc'=\{K\in \Hc(G): t_Kd'_K \text{ is not coprime to }d_K \text{ or } d_K<d\}.$$
For each $K\in \Hc'$ fix $b_K\in B_k(G)$ and $m_K\in \Z$ as above and $m_1=\lcm(m_K:K\in \Hc')$ and
    $$b_1=\left(\prod_{K\in \Hc'}b_K^{-\frac{m_1}{m_K}}\right)b^{m_1}.$$
Then $b_1\in B_k(G)$, $\rho_K(b_1)=1$ if $K\in \Hc'$ and $\rho_K(b_1)\in \GEN{\eta_k(\rho_K(a)^{t_K})}$ if $K\in \Hc(G)\setminus \Hc'$. Observe that $t_H=d'_H=u_H=1$ and hence $H\not\in \Hc'$. Therefore, $\rho_H(b_1)\in \GEN{\eta}$, because $\eta_k(\rho_H(a)^{t_H})=\eta$. To finish the proof we prove that $\Hc'=\Hc(G)\setminus \{H\}$. Suppose the contrary, that is, assume $K\in \Hc(G)\setminus\{H\}$ with $d_K\ge d$ and $\gcd(t_Kd'_K,d_K)=1$. The latter implies that $d'_K=1$, or equivalently $G=\GEN{a,K}$, and $t_K=[KH:K]$ is coprime to $d_K=[G:K]$. Consequently, $t_K=1$, or equivalently $H\subseteq K$. Hence, the assumption $d_K=[G:K]\ge [G:H]=d$ implies that $H=K$, a contradiction.
\end{proof}

As it was mentioned in the introduction, the Bass-Milnor Theorem is equivalent to Proposition~\ref{OneComponentBass} and the well known fact that the cyclotomic units generate a subgroup of finite index of $\U(\Z[\zeta])$ for every root of unity $\zeta$. We include a proof for completeness.

\begin{theorem}[Bass-Milnor]\label{BassAbelianFiniteIndexT}
If $G$ is a finite abelian group then the group generated by the Bass units of $\Z G$ has finite index in $\U(\Z G)$.
\end{theorem}

\begin{proof}
Let $\Hc=\Hc(G)$, $f=\prod_{H\in \Hc} \rho_H:\Q G\rightarrow \prod_{H\in \Hc} \Q(\zeta_{[G:H]})$, $B=B(G)$ and $V$ be the subgroup of $\prod_{H\in \Hc} \U(\Z(\zeta_{[G:H]}))$ generated by the units that project on one component to a cyclotomic unit and project on all other components to 1. As the cyclotomic units of each ring of cyclotomic integers $\Z[\zeta_n]$ generate a subgroup of finite index in $\U(\Z[\zeta_n])$, $V$ has finite index in $\prod_{H\in \Hc} \U(\Z(\zeta_{[G:H]}))$. Hence, by Lemma \ref{BassProjectionsL}, $f(B)$ has finite index in $\prod_{H\in \Hc} \U(\Z(\zeta_{[G:H]}))$ and therefore $B$ has finite index in $\U(\Z G)$, since $B \subseteq \U(\Z G) \subseteq f\inv(\prod_{H\in \Hc} \U(\Z(\zeta_{[G:H]})))$.
\end{proof}

We now present Algorithm \texttt{CyclotomicAsProductOfBass} which, for a cyclotomic unit $\eta$ of $\Q G$ as input, returns  a power of $\eta$ as product of Bass units. More precisely, the input is formed by a finite abelian group $G$, a subgroup $H\in \Hc(G)$, a list of linear characters $(\rho_K:K\in \Hc(G))$ of $G$ with $\ker \rho_K=K$ and two integers $k$ and $j$ with $\gcd(k,[G:H])=1$. The output is formed by an integer $m$, an integer $k'$ coprime to $|G|$ such that $k'\equiv k\mod [G:H]$ , and a list $((g_i,m_i,p_i):i\in I)$ with $g_i\in G$ and $m_i$ and $p_i$ integers such that $\gcd(k',|g_i|)=1$, $k'^{m_i}\equiv 1 \mod |g_i|$ and $b=\prod_{i\in I} \Bass{k'}{m_i}{g_i}^{p_i}$ is so that $\rho_H(b)=\eta_k(\zeta_{[G:H]}^j)^m$ and $\rho_K(b)=1$ for every $K\in \Hc(G)\setminus \{H\}$.

Observe that each linear character $\rho$ with kernel $K$ is determined by $\rho(a)$, where $a$ is an element of $G$ such that $G=\GEN{a,K}$. Moreover, for such an $a$ there is a unique linear character $\rho$ with kernel $K$ such that $\rho(a)=\zeta_{[G:K]}$. Thus we describe the list of characters $\rho_K$ by a list $A=(a_K:K\in \Hc(G))$ of elements of $G$ such that $G=\GEN{a_K,K}$ for every $K\in \Hc(G)$. In other words, in the input of \texttt{CyclotomicAsProductOfBass} we replace the list of characters by the list $A$. Then $\rho_K$ is the unique linear character of $G$ with kernel $K$ such that $\rho_K(a_K)=\zeta_{[G:K]}$.

\texttt{CyclotomicAsProductOfBass} follows the structure of the proof of Lemma~\ref{BassProjectionsL}. Notice that this is an inductive proof and henceforth the structure of the Algorithm is recursive. That is, the algorithm includes calls to itself. In these calls to itself the first argument, the ambient group, is replaced by a subgroup $S$ of $G$. In principle the algorithm could start by calculating $\Hc(G)$. However, this would make the algorithm inefficient because each call to itself should calculate $\Hc(S)$, which is not necessary because
    \begin{equation}\label{HcSHcG}
    \Hc(S)=\{S\cap K:K\in \Hc(G)\},
    \end{equation}
as was observed in the proof of Lemma~\ref{BassProjectionsL}. Thus we include $\Hc(G)$ as part of the input of the algorithm and, before each recursive call of the algorithm for a subgroup $S$ of $G$, one calculates $\Hc(S)$ from $\Hc(G)$ using (\ref{HcSHcG}). Of course one should filter the list to eliminate repetitions in $\Hc(S)$.

Furthermore, from the list $\{a_K:K\in \Hc(G)\}$, satisfying $G=\GEN{a_K,K}$ for each $K\in \Hc(G)$, one can calculate another list $\{b_Y:Y\in \Hc(S)\}$, satisfying $S=\GEN{b_Y,Y}$ for each $Y\in \Hc(S)$. In fact, the proof of Lemma~\ref{BassProjectionsL} tell us that for each $Y\in \Hc(S)$ one can select a $K_Y\in \Hc(G)$ such that $Y=S\cap K_Y$, so that if $Y\in \Hc(G)$ then $K_Y=Y$ and the linear representation $\lambda_Y$ of $S$ with kernel $Y$ chosen is the restriction of $\rho_{K_Y}$ to $S$.
Thus, we need to take $b_Y$ such that $\rho_K(b_Y)=\zeta_{[S:Y]}$.
In fact in each recursive call of the algorithm either $S=G$ or $[G:S]$ is prime.
If $S=G$, then $K_Y=Y$ and we may take $b_Y=a_Y$ for each $Y$. Assume otherwise that $[G:S]=p$, a prime integer.
If $Y\in \Hc(G)$ then $Y=K_Y$, $[S:Y]=\frac{[G:K_Y]}{p}$ and $\rho_{K_Y}(a_Y^p)=\zeta_{[G:K_Y]}^p = \zeta_{[S:Y]}$, so we may take $b_Y=a_Y^p$. Otherwise, i.e. if $Y\not\in \Hc(G)$, then $[K_Y:Y]=p$ and $G/Y=S/Y\times K_Y/Y$.
In particular, $p$ divides $[G:Y]$ and $[S:Y]=[G:K_Y]$. Then $b_Y$ is an element of $a_{K_Y}K_Y\cap S$.

We use the following notation for $m,n\in \Z$ and $A$, $B$ and $A_l$ ($l\in L$) lists.
\begin{eqnarray*}
m \mod n &=& \text{Remainder of }m \text{ modulo }n; \\
O_n(m) &=& \text{Multiplicative order of } m \text{ modulo }n \text{ (if } \gcd(m,n)=1);\\
(\_) &=& \text{Empty list};\\
A\sqcup B &=& \text{ Concatenation of the lists } A \text{ and } B; \\
\bigsqcup_{i=1}^l A_i &=& \text{ Concatenation of the lists } A_1,\dots,A_l.
\end{eqnarray*}

{\small
\begin{algorithm}[h!]
 \caption{\texttt{CyclotomicAsProductOfBass}($G,\Hc(G),A,H,k,j$)} \label{alg1}
\begin{algorithmic}[1]
 \Require $G$ a finite abelian group,

 $\Hc(G)=(H\leq G  :  G/H \mbox{ cyclic})$,

 $A=(a_K :  K\in\Hc(G))$, a list of group elements $a_K$ of $G$ such that $G=\GEN{a_K,K}$,

 $H\in \Hc(G)$,

 $j,k\in \N$ with $\gcd(k,[G:H])=1$.
\Comment{$\rho_K$ linear representation of $G$: $\ker \rho_K = K$; $\rho_K(a_K)=\zeta_{[G:K]}$.}

 \Ensure $(m, k', B=((g_i,m_i,q_i)  :  i=1,\dots,r)) \in \Z^2 \times (G\times \Z^{2})^{r}$.

 \Comment{
 $\rho_K\left(\prod_{i=1}^{r}\Bass{k'}{m_i}{g_i}^{q_i}\right)=
 \left\{\begin{array}{ll} \eta_k(\zeta_{[G:H]}^j)^m, & \text{if } K=H;\\ 1, &\text{if } K\in \Hc(G)\setminus \{H\}.
 \end{array}\right.$}

\State $d := [G:H]$; 
\State $k':=$ an integer coprime to $|G|$ congruent to $k$ modulo $d$.
    \Comment{Chinese Remainder Theorem.}
\State $j:=(j \mod d);$

\If {$d=1$}
    \State $m:=1;\;B=(\_)$;

\ElsIf {$\gcd(j,d)\neq 1$}
        \State $p :=$ common prime divisor of $j$ and $d$;
        \State $S:=\GEN{a_H^p,H}$ \Comment{Unique subgroup of $G$ of index $p$ containing $H$.}

        \State $\Hc(S) := (\_); \; A_S=(\_)$;
        \For {$K \in \Hc(G)$ with $K\subseteq S$}
            \State Add $K$ to $\Hc(S)$;
            \State Add $a_{K}^p$ to $A_S$;
        \EndFor
        \For {$K\in \Hc(G)$ with $K\not\subseteq S$}
            \State $Y:=K\cap S$;
            \If {$Y \not\in \Hc(S)$} \Comment{To exclude redundancy}
                \State Add $Y$ to $\Hc(S)$;
                \State Add an element of $S\cap a_K K$, to $A_S$;
            \Comment{$Y=S \cap K=\ker \rho_K|_Y, \rho_K(b_Y)=\zeta_{[S:Y]}$.}
            \EndIf
        \EndFor

        \State $(m,k',B=((s_i,m_i,q_i):i\in L)):=\texttt{CyclotomicAsProductOfBass}(S,\Hc(S),A_S,H,k',j/p)$;

        \State $H_1:=\GEN{a_{H}^{d/p},H}$; \Comment{Only subgroup of $G$ containing $H$ with $[H_1:H]=p$}
        \For{$i\in L$}
            \State Determine $j_i$ such that $s_i\in a_{H_1}^{j_i}H_1$;
            \State $(n_i,k',C_i=((g_u,m_u,r_u) :  u\in L_i)) := \texttt{CyclotomicAsProductOfBass}(G,\Hc(G),A,H_1,k',j_i)$

        \EndFor

        \State $m'$ := $\lcm(n_i:i\in L)$;
        \State $B := ((s_i,m_i,q_im'):i\in L) \sqcup \bigsqcup_{i\in L} \left(\left(g_u,m_u,-\frac{r_uq_im_im'}{n_i}\right): u\in L_i\right)$;
        \State $m$ := $mm'$;

    \Else \Comment{if $\gcd(j,d)=1$}
        \State $m := \lcm (O_{|a_H^jh|}(k')  :  h\in H)$;
        \State $B := ((a_H^jh,m,1): h\in H)$; 
        \For{$K\in \Hc(G)\setminus\{H\}$}
            \State $d_K := [G:K]; \; h_K := |H\cap K|; \; t_K := [H:H\cap K]$;
            \State $v_K := \text{ integer such that } a_HK=a_K^{v_K}K$; 
            \Comment{$\rho_K\left(\prod_{h\in H} \Bass{k'}{m}{a_{H}^jh}\right) = \eta_k(\zeta_{d_K}^{t_Kv_K})^{mh_{H}}$.}
            \State \begin{tabular}{rcl}
            $B_K$&$=$&$(m_K,k',((g_{i_K},m_{i_K},r_{i_K}) :  i_K \in I_K))$ \\ &:=& \texttt{CyclotomicAsProductOfBass}($G,\Hc(G),A,K,k',t_Kv_K$);
            \end{tabular}

        \EndFor

        \State $n$ := $\lcm(m_K:K\in \Hc(G)\setminus\{H\})$;
        \State $B:= ((a_H^jh,m,n) :  h\in H) \sqcup \bigsqcup_{K\in \Hc(G)\setminus \{H\}} \left(\left(g_{i_K},m_{i_K},-\frac{r_{i_K}mh_Kn}{m_K}\right): i_K \in I_K\right)$;
        \State $m$ := $mn|H|$;

\EndIf

    \Return $(m,k',B)$;

\end{algorithmic}
\end{algorithm}}\clearpage

This algorithm is suitable for implementation in an appropriate programming language such as GAP \cite{GAP4} and could be added to existing GAP packages such as Wedderga \cite{Wedderga}, which also deals with computations in group rings.

\section{A basis of Bass units}\label{SectionBasis}

Bass proved that if $G=\GEN{g}$, a cyclic group of order $n$, and $m$ is a multiple of $\varphi(n)$ then $\left\{\Bass{k}{m}{g^d}:d\mid n,1<k< \frac{n}{2d}, (k,\frac{n}{d})=1\right\}$ is a basis of a free abelian subgroup of finite index in $\U(\Z G)$ \cite{bass1966}. (Here $\varphi$ stands for the Euler totient function.) In this section we generalize this result and obtain a basis of Bass units for a subgroup of finite index in the integral group ring of an arbitrary abelian group $G$.
Moreover, the proof provides, for an arbitrary Bass unit $b$, an algorithm to express a power of $b$ as a product of a trivial unit and powers of at most two units in this basis of Bass units.

\begin{theorem}\label{BasisBassT}
Let $G$ be a finite abelian group. For every cyclic subgroup $C$ of $G$ choose a generator $a_C$ of $C$ and for every $k$ coprime to the order of $C$ choose an integer $m_{k,C}$ with $k^{m_{k,C}}\equiv 1 \mod |C|$. Then
    $$\left\{\Bass{k}{m_{k,C}}{a_C}:C \text{ cyclic subgroup of } G, 1<k< \frac{|C|}{2}, \gcd(k,|C|)=1\right\}$$
is a basis for a free abelian subgroup of finite index in $\U(\Z G)$.
\end{theorem}

\begin{proof}
The proof is based on the following equalities (\cite[Lemma 3.1]{2006GoncalvesPassman}):
\begin{eqnarray}
 \label{Basseq1} \Bass{k}{m}{g}&=&\Bass{k_1}{m}{g}, \mbox{ if } k\equiv k_1 \mod |g|, \\
 \label{Basseq2} \Bass{k}{m}{g}\Bass{k}{m_1}{g}&=&\Bass{k}{m+m_1}{g},\\
 \label{Basseq3} \Bass{k}{m}{g}\Bass{k_1}{m}{g^k}&=&\Bass{kk_1}{m}{g},\\
 \label{Basseq4} \Bass{1}{m}{g}&=&1 \text{ and } \\
 \label{Basseq5} \Bass{|g|-1}{m}{g} &=& (-g)^{-m},
\end{eqnarray}
for $g\in G$ and $k^m\equiv k_1^m \equiv k^{m_1} \equiv 1 \mod |g|$. By (\ref{Basseq2}) we have
\begin{equation}\label{Basseq6}
\Bass{k}{m}{g}^{h} = \Bass{k}{mh}{g}
\end{equation}
and from
(\ref{Basseq1}), (\ref{Basseq3}) and (\ref{Basseq5}) we deduce
    \begin{equation}\label{Basseq7}
    \Bass{|g|-k}{m}{g} = \Bass{k(|g|-1)}{m}{g} = \Bass{k}{m}{g} \Bass{|g|-1}{m}{g^k} = \Bass{k}{m}{g}g^{-km}
    \end{equation}
provided $(-1)^m\equiv 1 \mod |g|$.

By Theorem~\ref{BassAbelianFiniteIndexT}, the Bass units generate a subgroup of finite index in $\U(\Z G)$. Let $t=\varphi(|G|)$.
We first prove that $B_1=\left\{\Bass{k}{t}{a_C}\mid 1<k<\frac{|C|}{2}, \gcd(k,|C|)=1\right\}$ generates a subgroup of finite index in $\U(\Z G)$.
To do so we ``sieve'' gradually the list of Bass units, keeping the property that the remaining Bass units still generate a subgroup of finite index in $\U(\Z G)$, until the remaining Bass units are the elements of $B_1$. By equation (\ref{Basseq1}), to generate $B(G)$ it is enough to use the Bass units of the form $\Bass{k}{m}{g}$ with $g\in G$, $1\le k < |g|$ and $k^m \equiv 1 \mod |g|$. By (\ref{Basseq6}), for every Bass unit $\Bass{k}{m}{g}$ we have $\Bass{k}{m}{g}^u = \Bass{k}{t}{g}^v$ for some positive integers $u$ and $v$. Thus the Bass units of the form $\Bass{k}{t}{g}$ with $1\le k < |g|$ and $\gcd(k,|g|)=1$ generate a subgroup of finite index in $\U(\Z G)$. By (\ref{Basseq2}) and (\ref{Basseq3}), we can reduce further the list of generators by taking only those with $g=a_C$ for some cyclic group $C$ of $G$. By (\ref{Basseq4}) and (\ref{Basseq5}) we can exclude the Bass units with $k=\pm 1$ and still generate a subgroup of finite index in $\U(\Z G)$ with the remaining elements. Finally, $\Bass{k}{t}{g}\inv \Bass{|g|-k}{t}{g}=\Bass{|g|-1}{t}{g^k}=g^{-kt}$, by (\ref{Basseq5}) and (\ref{Basseq7}).  Thus $\Bass{k}{t}{g}\inv \Bass{|g|-k}{t}{g}$ has finite order. Therefore the units $\Bass{k}{t}{g}$ with $k>\frac{|g|}{2}$ can be excluded. The remaining units are exactly the elements of $B_1$. Thus $\GEN{B_1}$ has finite index in $\U(\Z G)$, as desired.

Let $\mathcal{C}$ be the set of cyclic subgroups of $G$ and $B=\{\Bass{k}{m_{k,C}}{a_C} : C\in \mathcal{C}, 1<k< \frac{C}{2}, \gcd(k,|C|)=1\}$. Using (\ref{Basseq6}) once more, we deduce that $\GEN{B}$ has finite index in $\U(\Z G)$, since so does $\GEN{B_1}$. To finish the proof we need to prove that the elements of $B$ are multiplicatively independent. To do so, it is enough to show that the rank of $\U(\Z G)$ coincides with the cardinality of $B$. For this, first observe that the cardinality of $B$ is $\sum_d k_d t_d$, where $d$ runs through the divisors of $|G|$, $k_d$ is the number of cyclic subgroups of $G$ of order $d$ and $t_d$ is the cardinality of $\{k :  1<k< \frac{d}{2}, \gcd(d,k)=1\}$. Obviously $t_1=t_2=0$ and $t_d=\frac{\varphi(d)}{2}-1$ for every $d>2$. Therefore, $|B|=\sum_{d>2} \left(\frac{k_d\varphi(d)}{2}-k_d\right) = \frac{1+k_2+\sum_d h_d}{2}-\sum_d k_d = \frac{1+k_2+|G|-2c}{2}$, where $h_d$ denotes the number of elements of $G$ of order $d$ (so that $h_1=1$ and $h_2=k_2$) and $c$ is the number of cyclic subgroups of $G$. By a Theorem of Higman \cite{Higman1940}, this number coincides with the rank of $\U(\Z G)$ and the proof is finished.
\end{proof}

In Theorem~\ref{BasisBassT} one can choose, for example, $m_{k,C}=\varphi(|G|)$, $m_{k,C}=\varphi(|C|)$ or $m_{k,C} = O_{|C|}(k)$. Observe that the Bass Theorem is the specialization of Theorem~\ref{BasisBassT} to $G=\GEN{g}$, $a_{\GEN{g^d}}=g^d$ for $d$ dividing $|g|$, and $m_{k,\GEN{g^d}}$ a fixed multiple of $\varphi(|g|)$. On the other hand, using the Bass and the Bass-Milnor Theorem one can easily prove that the units of Theorem~\ref{BasisBassT} generate a subgroup of finite index in $\U(\Z G)$. Indeed, by the Bass-Milnor Theorem, the Bass  units generate a subgroup of finite index of $\U(\Z G)$. On the other hand, by the Bass Theorem, if $\GEN{g}=C=\GEN{a_C}$, then a power of $\Bass{k}{m}{g}$ belongs to the group generated by the units of the form $\Bass{l}{\varphi(|a_C^j|)}{a_C^j}$, with $1 < l < \frac{|a_C^j|}{2}$ and $\gcd(l,|a_C^j|)=1$. Thus another power belongs to the group generated by the units of the form $\Bass{l}{m_{l,C}}{a_C^j}$.

The advantage of the new proof, with respect to the proofs of Bass and Bass-Milnor, is that it provides a way to express some power of any given Bass unit $\Bass{k}{m}{g}$ as a product of a trivial unit and powers of at most 2 elements from
    $$B=\left\{\Bass{k}{m_{k,C}}{a_C}:1<k\le\frac{|C|}{2}, \gcd(k,|C|)=1, C\in \mathcal{C}\right\}$$
for any given choice of generators $a_C$ of cyclic subgroups and integers $m_{k,C}$ as in Theorem~\ref{BasisBassT}. This is obtained as follows: calculate

$\bullet$ $n:=|g|; \; C:=\GEN{g};$

$\bullet$ $k'_1:=$ the unique integer $0\le k'_1 < n$ such that $g=a_C^{k'_1};$

$\bullet$ $k'_0:= kk'_1 \mod n$;

$\bullet$ for $i=0,1: \;
    k_i := \min(k'_i,n-k'_i); \;
    h_i := \left\{\begin{array}{ll} 1, & \text{if } k_i=k'_i; \\ a_C^{k'_i}, & \text{otherwise}.

    \end{array}\right.$

$\bullet$ $M:=\lcm\left(m,m_{k_0,C},m_{k_1,C}\right); \; c := \frac{M}{m};$

Then, by (\ref{Basseq3}), (\ref{Basseq6}) and (\ref{Basseq7}) we have
    \begin{eqnarray*}
    \Bass{k}{m}{g}^c &=& \Bass{k}{M}{a_C^{k'_1}} = \Bass{k'_0}{M}{a_C} \Bass{k'_1}{M}{a_C}^{-1} = \Bass{k_0}{M}{a_C} h_0^M \Bass{k_1}{M}{a_C}^{-1} h_1^{-M} \\
    &= & (h_0h_1\inv)^M \Bass{k_0}{m_{k_0,C}}{a_C}^{\frac{M}{m_{k_0,C}}} \Bass{k_1}{m_{k_1,C}}{a_C}^{-\frac{M}{m_{k_1,C}}}.
    \end{eqnarray*}

We summarize this result in the following corollary.
\begin{corollary}\label{BassInBasis}
Let $G$ be a finite abelian group. For every cyclic subgroup $C$ of $G$ choose a generator $a_C$ of $C$ and for every $k$ coprime to the order of $C$ choose an integer $m_{k,C}$ with $k^{m_{k,C}}\equiv 1 \mod |C|$. Then
    $$\left\{\Bass{k}{m_{k,C}}{a_C}:C \text{ cyclic subgroup of } G, 1<k< \frac{|C|}{2}, \gcd(k,|C|)=1\right\}$$
is a basis for a free abelian subgroup of finite index in $\U(\Z G)$. Moreover, for any Bass unit $\Bass{k}{m}{g}$ in $\Z G$ we have
    $$\Bass{k}{m}{g}^{c}=h\; \Bass{k_0}{m_{k_0,C}}{a_C}^{n_0}\; \Bass{k_1}{m_{k_1,C}}{a_C}^{n_1},$$
for $C=\GEN{g}$, an element $h\in G$ and integers $c,n_0,n_1,k_0,k_1$ so that $1\le k_0,k_1 \le \frac{|C|}{2}$, $g=a_C^{\pm k_1}$ and $k_0\equiv \pm kk_1 \mod |C|$.
\end{corollary}

Observe that the factor $\Bass{k_i}{m_{k_i,C}}{a_C}$ in Corollary~\ref{BassInBasis} is 1 if $k_i=1$. Otherwise it is a basis element. The integers $k_0$ and $k_1$ are uniquely determined by the conditions imposed, namely $1\le k_0,k_1 \le \frac{|C|}{2}$, $g=a_C^{\pm k_1}$ and $k_0\equiv \pm kk_1 \mod |C|$.  A possible choice for the integers $c,n_0$ and $n_1$ is $c=\frac{M}{m}$, $n_0=\frac{M}{m_{k_0,C}}$ and $n_1=-\frac{M}{m_{k_1,C}}$ with $M=\lcm(m,m_{k_0,C},m_{k_1,C})$.
For some particular choices of the $m_{k,C}$'s one can take the same exponent $c$ for every Bass unit $\Bass{k}{m}{g}$. For example, if $n$ is the exponent of $G$ then one could choose a divisor of $\varphi(n)$ for each $m_{k,C}$. With this choice one can take $c=\varphi(n)$, $n_0= m\frac{\varphi(n)}{m_{k_0,C}}$ and $n_1= -m\frac{\varphi(n)}{m_{k_1,C}}$.

\section{A basis formed by products of Bass units which are powers of cyclotomic units}

In this section $p$ is a prime integer. We obtain a set of multiplicatively independent units that generate a subgroup of finite index in $\U(\Z G)$, for $G$ a $p$-group which is either elementary abelian or cyclic. Each of these units is both a product of Bass units and a power of a cyclotomic unit.
In both cases we will use the fact that $\{\eta_k(\zeta_{p^n})\mid 1<k< \frac{p^n}{2}, p\nmid k\}$ generates a free abelian subgroup of finite index in $\U(\Z[\zeta_{p^n}])$ for each $n\ge 1$ (see \cite[Theorem 8.2]{1982Washington}).

\begin{proposition}\label{ElementaryAbelian}
Let $p$ be a prime integer and $G$ an elementary abelian $p$-group. For every $H\in \Hc(G)$ fix an $a_H\in G$ with $G=\GEN{a_H,H}$. Then the set
    $$\left\{x_{k,H}=\prod_{h\in H}\Bass{k}{O_p(k)}{a_Hh} :  1<k<\frac{p}{2}, H\in \Hc(G)\setminus\{G\}\right\}$$
is multiplicatively independent, generates a free abelian subgroup of finite index in $\U(\Z G)$ and each of its elements $x_{k,H}$ is a power of a cyclotomic unit of $\Q G$ and it belongs to $B(G)$.
\end{proposition}

\begin{proof}
Clearly $\rho_G(\Z G)=\Z$ and if $H\in \Hc(G)\setminus \{G\}$ then $\rho_H(\Z G) = \Z[\zeta_p]$.
Let $K\in \Hc(G)$.
By Lemma~\ref{ProductBassCyclic}, $\rho_K(x_{k,H}) = \eta_k(\rho_K(a_H)^{[K:H\cap K]})^{O_p(k)|H\cap K|}$. Furthermore, $\rho_H(a_H)=\zeta_p$ and, if $K\ne H$ then $[K:H\cap K]=p$ and hence $\rho_K(a_H)^{[K:H\cap K]}=1$. Thus
    $$\rho_K(x_{k,H}) = \left\{\begin{array}{ll}
    \eta_k(\zeta_p)^{O_p(k)|H|}, & \text{if } K=H;\\ 1, & \text{otherwise}.
    \end{array}\right.$$
Since $\eta_k(\zeta_p):1<k<\frac{p}{2}$ are multiplicatively independent and generate a subgroup of finite index in $\U(\Z[\zeta_p])$, the result follows.
\end{proof}

\begin{proposition}\label{CyclicpGroup}
Let $G=\GEN{a}$ be a cyclic group of order $p^n$ with $p$ prime. For every $i=0,1,\dots,n$, let $G_i=\GEN{a^{p^{n-i}}}$, the subgroup of $G$ of order $p^i$.
Let $k$ be a positive integer coprime to $p$. For every $0\le j \le s \le n$ we construct recursively the following products of Bass units of $\Z G$:
$$b_s^s(k)=1,$$ and, for $0\le j\le s-1$,
$$b_j^s(k)= \left(\prod_{h\in G_j}\Bass{k}{O_{p^n}(k)}{a^{p^{n-s}}h}\right)^{p^{s-j-1}} \left( \prod_{i=j+1}^{s-1}b_i^s(k)\inv \right) \left(\prod_{i=0}^{j-1} b_i^{s+i-j}(k)\inv\right).$$
Then
    $$\left\{b_j^n(k) :  1<k<\frac{p^{n-j}}{2}, 0\le j\le n, p\nmid k\right\}$$
is multiplicatively independent, generates a free abelian subgroup of finite index in $\U(\Z G)$ and each $b_j^{n}(k)$ is a power of a cyclotomic unit of $\Q G$ and belongs to $B(G)$.
\end{proposition}

\begin{proof}
Following the same strategy as in the proof of Proposition~\ref{ElementaryAbelian} it is enough to prove the following
\begin{equation}\label{rhobs}
\rho_{G_{j_1}}(b_j^s(k)) = \left\{
\begin{array}{ll}
\eta_k(\zeta_{p^{s-j}})^{O_{p^n}(k)p^{s-1}}, & \text{if } j=j_1;\\
1, & \text{if } j\neq j_1.
\end{array}\right.
\end{equation}
for every $0\le j,j_1\le s\le n$.
Again we use a double induction, first on $s$ and, for fixed $s$, on $s-j$ (roughly said, $s$ is a replacement for (IH1) and $s-j$ for (IH2)). The minimal cases ($s=0$ and $s=j$) are obvious, so assume that $s>1$, $j<s$ and (\ref{rhobs}) holds when $s$ is replaced by a smaller value $s_1$ and any $0\le j,j_1\le s_1$ and when, with $s$ fixed, $j$ is replaced by a bigger value.

Let $u(k)=\prod_{h\in G_j}\Bass{k}{O_{p^n}(k)}{a^{p^{n-s}}h}$. By Lemma~\ref{ProductBassCyclic}, if $0\le i \le j$ then
    \begin{equation}\label{CyclicpGroup1}
    \rho_{G_i}(u(k)) = \eta_k(\rho_{G_i}(a^{p^{n-s}})^{p^{j-i}})^{O_{p^n}(k)p^i} = \eta_k(\zeta_{p^{s-i}}^{p^{j-i}})^{O_{p^n}(k)p^i},
    \end{equation}
and if $j\le i\le n$ then
    \begin{equation}\label{CyclicpGroup1b}
    \rho_{G_i}(u(k)) = \eta_k(\rho_{G_i}(a^{p^{n-s}}))^{O_{p^n}(k)p^j} = \eta_k(\zeta_{p^{s-i}})^{O_{p^n}(k)p^j}.
    \end{equation}
By (\ref{CyclicpGroup1}) and the induction hypothesis on $s$, we have
\begin{equation}\label{CyclicpGroup2}
\rho_{G_{j_1}}(b_i^{s+i-j}(k)) = \left\{
\begin{array}{ll}
\eta_k(\zeta_{p^{s-j}})^{O_{p^n}(k)p^{s+i-j-1}} = \rho_{G_{j_1}}(u(k))^{p^{s-j-1}},& \text{if } j_1=i<j;\\
1, & \text{if } j_1\ne i<j.
\end{array}\right.
\end{equation}
By (\ref{CyclicpGroup1b}) and the induction hypothesis on $s-j$, we have
\begin{equation}\label{CyclicpGroup3}
\rho_{G_{j_1}}(b_i^s(k)) = \left\{
\begin{array}{ll}
\eta_k(\zeta_{p^{s-i}})^{O_{p^n}(k)p^{s-1}}=\rho_{G_{j_1}}(u(k))^{p^{s-j-1}}, & \text{if } j<i=j_1;\\
1, & \text{if } j<i\neq j_1
\end{array}\right.
\end{equation}
Now, combining (\ref{CyclicpGroup1}), (\ref{CyclicpGroup2}) and (\ref{CyclicpGroup3}) we have
    $$\rho_{G_{j_1}}(b_j^s(k)) = \left\{
    \begin{array}{ll}
    \rho_{G_j}(u(k))^{p^{s-j-1}} = \eta_k(\zeta_{p^{s-j}})^{O_{p^n}(k)p^{s-1}}, & \text{if } j=j_1;\\
    \rho_{G_{j_1}}(u(k))^{p^{s-j-1}} \rho_{G_{j_1}}(b_{j_1}^s(k))^{-1} = 1,&  \text{if } j< j_1;\\
    \rho_{G_{j_1}}(u(k))^{p^{s-j-1}} \rho_{G_{j_1}}(b_{j_1}^{s+j_1-j}(k))^{-1} = 1,&  \text{if } j> j_1;
    \end{array}
    \right.$$
as desired.
\end{proof}

\renewcommand{\bibname}{References}
\bibliographystyle{amsalpha}
\bibliography{references}

\end{document}